\documentclass[10pt]{amsart}
\usepackage[leqno]{amsmath}
\usepackage{amsthm}
\usepackage{amsfonts}
\usepackage{amssymb}
\usepackage{amsmath,graphicx}
\usepackage{eucal}
\usepackage{ stmaryrd }
\usepackage{xcolor}
\usepackage{dirtytalk}
\usepackage[all]{xy}
\CompileMatrices

\newcommand{\z}{\mathbb{Z}}
\newcommand{\zp}{\mathbb{Z}_p}
\newcommand{\q}{\mathbb{Q}}
\newcommand{\qp}{\mathbb{Q}_p}
\newcommand{\bbc}{\mathbb{C}}

\newcommand{\real}{\mathbb{R}}
\newcommand{\mm}{\mathfrak{m}}
\newcommand{\mmk}{\mathfrak{m}_K}
\newcommand{\mml}{\mathfrak{m}_L}
\newcommand{\mmc}{\mathfrak{m}_{\mathbb{C}_K}}
\newcommand{\mmp}{\mathfrak{m}_{\mathbb{C}_p}}
\newcommand{\fp}{\mathbb{F}_p}

\newcommand{\fqf}{\mathbb{F}_{q^f}}
\newcommand{\mup}{\mu_{[\pi^\infty]}}
\newcommand{\mupn}{\mu_{[\pi^n]}}
\newcommand{\px}{[\pi](X)}
\newcommand{\pn}{[\pi^n]}
\newcommand{\lb}{\left (}
\newcommand{\rb}{\right )}
\newcommand{\llb}{\llbracket}
\newcommand{\rrb}{\rrbracket}
\newcommand{\lpi}{\log_{[\pi]}}
\newcommand{\dpi}{\mathcal{D}_{K}}
\newcommand{\dpil}{\mathcal{D}_{K}(L)}

\newcommand{\llf}{\left \lfloor}
\newcommand{\rrf}{\right \rfloor}

\newcommand{\sm}{\setminus}

\newcommand{\fpf}{\mathbb{F}_{p^f}}
\newcommand{\Char}{\operatorname{char}}
\newcommand{\ltml}{\mathcal{F}(\mathfrak{m}_L)}
\newcommand{\ltmn}{\mathcal{F}(\mathfrak{m}_{K_{\pi^n}})}
\newcommand{\ltmli}{\mathcal{F}(\mml^i)}
\newcommand{\ltmlii}{\mathcal{F}(\mml^{i+1})}
\newcommand{\ltmlqi}{\mathcal{F}(\mml^{qi})}
\newcommand{\ltmlqii}{\mathcal{F}(\mml^{qi+1})}
\newcommand{\ltmlj}{\mathcal{F}(\mml^j)}
\newcommand{\ltmljj}{\mathcal{F}(\mml^{j+1})}
\newcommand{\ltmlqj}{\mathcal{F}(\mml^{qj})}
\newcommand{\ltmlqjj}{\mathcal{F}(\mml^{qj+1})}
\newcommand{\ltm}{\mathcal{F}(\mathfrak{m}_{\bbc_K})}
\newcommand{\kpn}{K_{\pi^n}}
\newcommand{\Fsum}{%
  \mathop{%
    \mathchoice
      {\ooalign{%
        $\displaystyle\sum$\cr
        \hidewidth\raisebox{-0.1in}{\hspace{0.3in}$\displaystyle F$}\hidewidth\cr
      }}%
      {\ooalign{%
        $\textstyle\sum$\cr
        \hidewidth\raisebox{-0.22in}{\hspace{0.08in}$\scriptstyle F$}\hidewidth\cr
      }}%
      {\ooalign{%
        $\scriptstyle\sum$\cr
        \hidewidth\raisebox{-0.14in}{\hspace{0.05in}$\scriptscriptstyle F$}\hidewidth\cr
      }}%
      {\ooalign{%
        $\scriptscriptstyle\sum$\cr
        \hidewidth\raisebox{-0.1in}{\hspace{0.03in}$\scriptscriptstyle F$}\hidewidth\cr
      }}%
  }\displaylimits
}

\makeatletter
\def\moverlay{\mathpalette\mov@rlay}
\def\mov@rlay#1#2{\leavevmode\vtop{%
   \baselineskip\z@skip \lineskiplimit-\maxdimen
   \ialign{\hfil$\m@th#1##$\hfil\cr#2\crcr}}}
\newcommand{\charfusion}[3][\mathord]{
    #1{\ifx#1\mathop\vphantom{#2}\fi
        \mathpalette\mov@rlay{#2\cr#3}
      }
    \ifx#1\mathop\expandafter\displaylimits\fi}
\makeatother

\theoremstyle{plain}
\newtheorem{theorem}{Theorem}[section]
\newtheorem{lemma}[theorem]{Lemma}

\newtheorem{cor}[theorem]{Corollary}

\theoremstyle{definition}
\newtheorem{remark}[theorem]{Remark}
\newtheorem{definition}[theorem]{Definition}
\newtheorem{example}[theorem]{Example}

\numberwithin{equation}{section}

\usepackage{url}
\usepackage[colorlinks=true, pdfstartview=FitV, linkcolor=blue, citecolor=blue, urlcolor=blue]{hyperref}

\usepackage{color}

\title{An $O_K$-basis for the image of a Lubin-Tate logarithm on $\pi$-regular extensions of $K$}
\author{Georgia Harbor-Collins}
\address{Department of Mathematics, 
Univ. of Connecticut, Storrs, CT 06269-3009}
\email{georgia.harbor-collins@uconn.edu}

\begin{document}
\begin{abstract}
Let $K$ be a finite $p$-adic field with uniformiser $\pi$. In this paper we study the image of the logarithm attached to a Lubin-Tate series $\px$ on the maximal ideal of so-called \linebreak $\pi$-regular extensions of $K$; for such an extension $L|K$ we compute a basis for the additive group $\lpi(\ltml)$ as an $O_K$-module, where $\ltml$ denotes the maximal ideal $\mml$ equipped with the $O_K$-module structure coming from the formal group associated to $\px$, and determine the minimal valuation of the elements in $\lpi(\ltml)$. In the final section of this paper we discuss how some of these results extend to arbitrary finite extensions of $K$ and conclude by determining a basis of the $O_K$-module $\lpi(\ltmn)$, where $\kpn$ is the Lubin-Tate extension of level $n\geq 1$.
\end{abstract}

\maketitle

\section{Introduction}

Let $K$ be a finite extension of $\qp$ with uniformiser $\pi \in O_K$ and residue field $k$ of size $q$. Fix a Lubin-Tate series $\px \in O_K \llb X\rrb$ at $\pi$ and denote by $\lpi(X)$ the associated logarithm. Let $\bbc_K$ be the completion of an algebraic closure of $K$ and denote by $\ltm$ the maximal ideal of $\bbc_K$ equipped with the $O_K$-module structure coming from the formal group associated to $\px$. The logarithm is a surjective homomorphism $\mathcal{F}(\mmc)\longrightarrow \bbc_K$. In contrast to that, the image of a Lubin-Tate logarithm when restricted to finite extensions of $K$ is bounded; if $L$ is a finite extension of $K$ then $\lpi:\ltmli \longrightarrow \mml^i$ is an isomorphism of $O_K$-modules for all $i>\displaystyle{\frac{v_L(\pi)}{q-1}}$ where $\ltmli$ denotes the set $\mml^i$ equipped with the $O_K$-module structure coming from the formal group associated to $\px$. In particular, $\lpi(\ltml)=\mml$ if $v_L(\pi)<q-1$.

When $v_L(\pi) \geq q-1$, not much is known about the image of a Lubin-Tate logarithm when restricted to the maximal ideal of a finite extension $L|K$. In the multiplicative case, i.e., when $\pi=p$ in $\qp$ and $[p](X)=(1+X)^p-1$, $\log_{[p]}(X)$ is the usual ($p$-adic) logarithm:
$$\log_{[p]}(X)=\log(1+X)=\sum_{n\geq 1}\frac{(-1)^{n-1}}{n}X^n.$$
In some cases the image of the $p$-adic logarithm is known; $\log(1+p\zp)$ is $4\z_2$ if $p=2$, and $p\zp$ otherwise. In \cite{2adic} and \cite{cycloimage} the image of the $p$-adic logarithm on principal units is worked out for each of the seven quadratic extensions of $\q_2$, as well as the cyclotomic extensions $\qp(\zeta_p)$ where $\zeta_p$ is a primitive $p$-th root of unity. In all of these examples the image of the logarithm sits entirely within the ring of integers $O_L$, which is not usually the case. 

In \cite{ontheimage}, the precise image of a Lubin-Tate logarithm on the maximal ideal is worked out for all unramified extensions $L|K$. In this paper we determine a basis for the additive group $\lpi(\ltml)$ as an $O_K$-module for so-called \textit{$\pi$-regular} extensions of $K$ (meaning that $L$ does not contain any nontrivial zeros of $\px$), and the Lubin-Tate extensions $\kpn$. In the $\pi$-regular case, we use this basis to give an exact lower bound on the elements of $v_L(\lpi(\ltml))$. 

 The idea will be to generalise that for finite extensions of $\qp$ containing no primitive $p$-th root of unity, we know a basis of the multiplicative group $1+\mml$ as a $\zp$-module \cite[Corollary 4.3.19]{cohen}. 
\begin{theorem}
    Let $L|\qp$ be a finite extension such that $L$ does not contain a non-trivial $p$-th root of unity, and denote by $e$ and $f$ the ramification index and residue field degree of $L|\qp$. Let $\varpi$ be a uniformiser of $L$ and let $\zeta_1,\ldots,\zeta_f$ in $O_L$ be a lift of an $\fp$-basis of the residue field of $L$. The $ef$ elements
$$1+\zeta_i\pi^j \text{ with } 1\leq i \leq f, \ 1\leq j \leq \frac{pe}{p-1} \text{ and } p\nmid j$$
constitute a $\zp$-basis of the multiplicative $\zp$-module $1+\mml$. 
\end{theorem}
In particular, since the logarithm is injective $1+\mml \longrightarrow \log(1+\mml)$ whenever $L$ does not\linebreak contain a primitive $p$-th root of unity, this basis of $1+\mml$ is mapped to a basis of $\log(1+\mml)$. In Section \ref{pireg} we generalise \cite[Corollary 4.3.19]{cohen} to give a basis of $\ltml$ as an $O_K$-module in the case that $L$ does not contain any nontrivial zeros of $\px$. By an analogous argument to before, this allows us to obtain a basis of the additive group $\lpi(\ltml)$ as an $O_K$-module. Here are our main results.
\begin{theorem}[Theorem \ref{basisthm2} below]\label{thm1}
    For a $\pi$-regular extension $L|K$ with uniformiser $\varpi$ and residue field degree $f$, the set of $[L:K]$ elements 
    $$  \widetilde{\mathcal{B}}_L:=\left\{\lpi(\zeta_i \varpi^j) \mid 1\leq i \leq f, \, 1\leq j \leq \frac{qv_L(\pi)}{q-1}, \, q \nmid j \right \}$$
is a basis of the additive group $\lpi(\ltml)$ as an $O_K$-module, where the $\zeta_i$'s in $O_L$ are any lift of a $k$-basis for $k_L$.
\end{theorem}
\begin{theorem}[Theorem \ref{minvalthm} below]\label{thm2}
    Let $L|K$ be $\pi$-regular and suppose $\displaystyle{\frac{v_L(\pi)}{q-1}}\geq 1$, so 
\begin{equation}\label{g.11}
    q^{\gamma-1}\leq \frac{v_L(\pi)}{q-1}<q^\gamma
\end{equation}
    for a unique integer $\gamma \geq 1$. Then, for any prime $\varpi$ of $L$,
    \begin{equation}\label{g.2}
      \min_{y \in \lpi(\ltml)}(v_L(y))=v_L(\lpi(\varpi))=q^{\gamma}-\gamma v_L(\pi).
    \end{equation}
    In particular, the smallest disc containing $\lpi(\ltml)$ is $\{x \in L \mid v_L(x) \geq q^{\gamma}-\gamma v_L(\pi)\}$.
\end{theorem}
As a consequence of Theorem \ref{thm2} we note that for a $\pi$-regular extension $L|K$
$$\lpi(\ltml) \subseteq O_L \iff q^{\gamma}-\gamma v_L(\pi) \geq 0 \iff v_L(\pi)\leq\frac{q^\gamma}\gamma ,$$
where $\gamma$ is defined as in Theorem \ref{thm2}. Explicitly, $\displaystyle{\gamma:=\llf \log_q\lb\frac{v_L(\pi)}{q-1}\rb\rrf+1.}$

Whilst $\pi$-regular extensions are the primary focus of this paper, Section \ref{extension} discusses how some of our results generalise to arbitrary ramified extensions $L|K$; if $L$ is a finite extension containing nontrivial zeros of $\px$ then $\ltml$ contains torsion and in particular, does not admit an $O_K$-basis. In Theorem \ref{genspanthm} we compute a spanning set for $\ltml$ as an $O_K$-module, that is \textit{minimal} for totally ramified extensions and use this to deduce the following. 
\begin{theorem}[Theorem \ref{ltbasisthm} below]\label{thm3}
       Fix $n\geq 1$ and a generator $\lambda$ of the $O_K$-module $\mupn$.  The set of $q^n-q^{n-1}$ elements 
    $$\mathcal{B}_n:=\{\lpi(\lambda^j) \mid 2\leq j \leq q^n-1, \ q\nmid j\} \cup \{\lpi(\lambda^{q^n})\}$$
    is a basis of the additive group $\lpi(\ltmn)$ as an $O_K$-module. 
\end{theorem}

The paper is organised as follows. Sections \ref{notations} and \ref{background} provide notations and the necessary background from Lubin-Tate theory. Section \ref{pireg} introduces the notion of a $\pi$-regular extension of $K$ and presents an $O_K$-basis of $\ltml$ in Theorem \ref{basisthm1}. Section \ref{main} contains the proofs of Theorems \ref{thm1} and \ref{thm2} and concludes with an example in the usual multiplicative setting over $\qp$. Finally, Section \ref{extension} discusses how some of these results generalise to finite extensions $L|K$ containing nontrivial zeros of $[\pi](X)$, and concludes with the proof of Theorem \ref{thm3}.

{\sc Acknowledgements}. I'd like to thank my advisor, Keith Conrad, for his helpful suggestions and thoughtful feedback.

\section{Notations}\label{notations}
\begin{tabular}{l|l|l}
Notation & Meaning & Equation \\
\hline
$ K$ & a finite extension of $\qp$ with prime element $\pi$ and residue field $k \cong \mathbb{F}_q$ & \\
$L$ & a finite extension of $K$ with maximal ideal $\mml=\varpi O_L$ and residue field $k_L \cong \fqf$ \\
$v_K$, $v_L$ & the normalised valuations on $K$ and $L$, respectively & \\ 
$\px$ & a Lubin-Tate series for the prime $\pi$ in $O_K$ & (\ref{ltseries}) \\
$F(X,Y)$ & the unique formal group law commuting with $\px$ & (\ref{fgl1})-(\ref{fgcomwithlt}) \\
$\ltml$ & the $O_K$-module $\mml$ with the Lubin-Tate module structure based on $F(X,Y)$& \\
    $\mu_{[\pi^n]}$ & the $[\pi^n]$-torsion contained in a fixed algebraic closure of $K$ & (\ref{mun}) \\
    $\kpn$ & the Lubin-Tate extension of level $n\geq 1$ obtained by adjoining $\mupn$ to $K$ & \\
    $\mu_{[\pi^\infty]}$& all the $[\pi]$-power torsion contained in a fixed algebraic closure of $K$& (\ref{mupandl}) \\
     $\mu_{[\pi^\infty]}(L)$& all the $[\pi]$-power torsion contained in a field extension $L|K$ & (\ref{mupandl}) \\
     $\lpi(X)$ & the formal logarithm associated to $\px$ & (\ref{logcharacteristics}) \\
     $\dpi$ & the disc in $\bbc_K$ on which the formal exponential associated to $\px$ converges & (\ref{dpi})
     \\
     $\dpil$ & the disc in $L$ on which the formal exponential associated to $\px$ converges & (\ref{dpil})     
\end{tabular}

\section{Background}\label{background}
Let $K$ be a nonarchimedean local field with ring of integers $O_K$ and maximal ideal $\mmk$. Let $q$ be the number of elements in the residue field $k:=O_K/\mmk$. Fix a generator $\pi$ of the maximal ideal $\mm_K$ and a Lubin-Tate series $\px$, that is, a formal power series in $O_K\llb X \rrb$ satisfying the two conditions
\begin{equation}\label{ltseries}
    \px \equiv \pi X \bmod{X^2} \text{ and } \px \equiv X^q  \bmod{\mmk}.
 \end{equation}
 In \cite[Theorem 1]{1965}, Lubin and Tate show there is a unique formal power series $F(X,Y)$ with coefficients in $O_K$ satisfying the following four properties:
 \begin{align}
 \label{fgl1}    F(X,Y) &\equiv X + Y \bmod{\deg 2}, \\
 \label{fgl2}    F(X,Y) &= F(Y,X), \\
 \label{fgl3}    F(X,F(Y,Z)) &= F(F(X,Y),Z), \\
  \label{fgcomwithlt}
    [\pi](F(X,Y))&=F([\pi](X),[\pi](Y)).
 \end{align}
We call $F(X,Y)$ the \textit{formal group law associated to $\px$}. In \cite[Theorem 1]{1965}, Lubin and Tate show to each $a \in O_K$ there is a unique formal power series $[a](X) \in aX + X^2O_K\llb X \rrb$ that commutes with $\px$ under composition and satisfies
$$[ab] = [a] \circ [b] \text{ and } [a+b]=[a]+[b]$$
in $O_K\llb X \rrb$ for all $a, b \in O_K$. By the uniqueness property, $[\pi^n](X)$ is the $n$-fold composition of $\px$. Let $\bbc_K$ denote the completion of an algebraic closure of $K$, $O_{\bbc_K}$ be its valuation ring, and $\mmc$ be the maximal ideal of $O_{\bbc_K}$. Since $F(X,Y)$ and $[a](X)$ have coefficients in $O_K$ for each $a \in O_K$, $F(x,y)$ and $[a](x)$ converge whenever $x, y \in \mmc$ and take values in $\mmc$. Thus we can define a new $O_K$-module structure on $\mmc$ where addition is given by the formal group law $F(X,Y)$, and scalar multiplication by $a \in O_K$ is given by the action of $[a](X)$. When equipped with this Lubin-Tate structure, we will denote the $O_K$-module $\mmc$ by $\ltm$ and define $\ltmli$ similarly for a finite extension $L|K$ and $i \geq 1$. For finite $L|K$ with residue field $k_L$, there are $O_K$-module isomorphisms
\begin{equation}\label{isotores}
    \ltmli/\ltmlii\stackrel{\cong}{\longrightarrow} \mml^i/\mml^{i+1} \stackrel{\cong}{\longrightarrow} k_L\footnote{Here the quotient $\ltmli/\ltmlii$ is being taken with respect to the Lubin-Tate module structure, so a general element of this quotient looks like $x+_F\ltmlii$. In particular, $x+_F\ltmlii=y+_F\ltmlii$ if and only if $F(x,[-1](y)) \in \ltmlii$.}
\end{equation}
for each $i\geq 1$. The last map in (\ref{isotores}) is the map that sends the class of $x\varpi^i$ in $\ltmli/\ltmlii$ to $x\bmod{\mml}$. The first map in (\ref{isotores}) is the identity on coset representatives, i.e., $$x\varpi^i +_F \ltmlii \longmapsto x\varpi^i+\mml^{i+1}.$$ It is easy to check this is a well defined bijection. It is a homomorphism of $O_K$-modules: by (\ref{fgl1}),
$$u\varpi^i +_F v\varpi^i=F(u\varpi^i,v\varpi^i) \equiv u\varpi^i+v\varpi^i \bmod{\mml^{i+1}},$$
and similarly for each $a \in O_K$ we have
$[a](u\varpi^i) \equiv au\varpi^i \bmod{\mml^{i+1}}$ since $[a](X)=aX+ X^2(\cdots)$.

For each integer $n \geq 1$, define the $[\pi^n]$-torsion to be the (finite) set
\begin{equation}\label{mun}
    \mupn:=\left\{ \lambda \in \mathcal{F}(\mmc) \mid [\pi^n](\lambda)=0 \right\},
\end{equation}
and let $\mup$ denote the union of $\mu_{[\pi^n]}$ over all $n$. For a finite extension $L|K$, $\mup(L)$ will denote the set of all $[\pi]$-power torsion contained in $L$. Explicitly,
\begin{equation}\label{mupandl}
    \mup:=\bigcup_{n\geq 1}\mupn \ \text{ and } \ \mup(L):=\mup\cap L.
\end{equation}
There are analogies between $\mupn$ and the $p^n$-th roots of unity in $\bbc_K$, with $\mathcal{F}(\mmc)$ being analogous to the group $1+\mmp$ as a $\zp$-module: the set $\mupn$ is a cyclic $O_K$-module of size $q^n$, for $a \in O_K$ we have $[a](\lambda)=0$ for all $\lambda \in \mupn$ if and only if $a\equiv 0\bmod{\pi^n}$, and adjoining $\mupn$ to $K$ generates a totally ramified abelian extension of degree $q^{n-1}(q-1)$ over $K$, which we shall denote by $K_{\pi^n}$ and often refer to as the \textit{Lubin-Tate extension of level $n$}. If $\lambda$ is a generator of $\mupn$ as a cyclic $O_K$-module, then $\lambda$ is prime in $K_{\pi^n}$ \cite[Theorem 2]{1965}. 

When $\Char K=0$, the \textit{formal logarithm associated to $\px$} will be denoted by $\lpi(X)$ and this is characterised in $K\llb X \rrb$ by the two conditions 
\begin{equation}\label{logcharacteristics}
\lpi(X) \equiv X \bmod{X^2} \text{ and } \lpi(\px)=\pi\lpi(X). \footnote{More generally, whenever $F(X,Y)$ is a formal group law defined over a torsion-free ring of characteristic zero, there exists a unique formal power series $\log_F(X) \in K\llb X\rrb$ satisfying the two conditions $\log_F(X) \equiv X\bmod{X^2}$ and $\log_F(F(X,Y))=\log_F(X)+\log_F(Y)$.}
\end{equation}
The logarithm gives an $O_K$-module homomorphism from $\mathcal{F}(\mmc)$ to $\bbc_K$ \cite[Theorem 6.4, p. 132]{silverman}.
On a sufficiently small disc, the logarithm is a bijection; 
$\lpi(X)$ is an $O_K$-module isomorphism $\mathcal{F}(\dpi) \longrightarrow \dpi$ \cite[Lemma 4, p. 215]{lang} where
\begin{equation}\label{dpi}
    \dpi:=\left \{x \in \bbc_K \mid v_K(x)> \frac{1}{q-1}\right \}.
\end{equation}
Similarly for finite $L|K$, $\lpi:\mathcal{F}(\dpil) \longrightarrow \dpil$ is an isomorphism of $O_K$-modules where
\begin{equation}\label{dpil}
    \dpil:=\dpi \cap L.
\end{equation}
For example, when $K=\qp$, $\pi=p$, and $[p](X)=(1+X)^p-1$, $\log_{[p]}(X)$ is the usual $p$-adic logarithm with a shift of domain from $1+\mmp$ to $\mmp$:
$$x \mapsto\log(1+x):=\sum_{n\geq 1}(-1)^{n-1}\frac{x^n}{n}=x-\frac{x^2}{2}+\frac{x^3}{3}-\frac{x^4}{4}+\cdots$$
for $x \in \mmp$.
When restricted to the disc $1+p\zp$, the $p$-adic log is an isometry with image $p\zp$ if $p>2$. Whilst this isn't exactly true when $p=2$ due to the presence of the $2$-power root of unity $-1 \in 1+2\z_2$, the $2$-adic logarithm is an isometric group isomorphism $1+4\z_2 \longrightarrow 4\z_2$ and it maps $1+2\z_2=\pm(1+4\z_2)$ to $4\z_2$. The formal logarithm associated to a Lubin-Tate series $\px$ behaves similarly to the usual $p$-adic logarithm. Some of these similarities are described in the following lemmas.

\begin{lemma}\label{surj} 
Let $K$ be a nonarchimedean local field of characteristic zero and $\px$ be a Lubin-Tate series for a chosen prime $\pi \in O_K$. When $L|K$ is a finite extension, $\lpi(\ltml)$ is bounded in $L$, while $\lpi(\ltm)=\bbc_K$. 
\end{lemma}
\begin{proof}
We will first prove the lemma for the basic Lubin-Tate series $\px=X^q+\pi X$ and then we will show that the general case can be reduced to this via an isomorphism of formal group laws. Pick an element $z \in \bbc_K$. We want to show there is a $x \in \mm_{\bbc_K}$ such that $\lpi(x)=z$. Using an approach analogous to how one shows the $p$-adic log surjects from $\mmc$ onto $\bbc_K$, pick $n\gg 0$ such that $\pi^n z \in \dpi$. Since $\lpi$ is surjective $\mathcal{F}(\dpi) \longrightarrow \dpi$, there is a $y \in \dpi$\footnote{This is not a typo: as sets, $\mathcal{F}(\dpi)=\dpi$. We will often drop the `$\mathcal{F}$' notation when making a statement about which power of the maximal ideal an element lives in.} such that
\begin{equation}\label{zy1}
    \pi^nz=\lpi(y).
\end{equation}
We claim there is an $x \in \mm_{\bbc_K}$ such that $y=\pn(x)$. Granting this claim, we can substitute $\pn(x)$ in place of $y$ in (\ref{zy1}), use $O_K$-linearity to \say{pull out} the $\pi^n$ and divide through to get $z=\lpi(x)$. 

To prove the claim, we note that $\px=X^q+\pi X$ is a distinguished polynomial\footnote{A polynomial $a_nX^n+\cdots+ a_1X+a_0$ is distinguished at a prime ideal $\mathfrak{p}$ if $a_n=1$ and $a_i \in \mathfrak{p}$ for all $0\leq i \leq n-1$. It is \say{almost Eisenstein}. } and in particular any iterate of it (such as $\pn(X)$) will also be distinguished. Consider the polynomial \linebreak $h(X):=\pn(X)-y$. Since $y \in \dpi \subseteq \mm_{\bbc_K}$, the polynomial $h(X)$ is distinguished and thus any root of $h(X)$ will lie in $\mm_{\bbc_K}$. Since $\bbc_K$ is algebraically closed, there is an $x \in \bbc_K$ (and necessarily $x \in \mm_{\bbc_K}$) such that $h(x)=0$, so $\pn(x)=y$, which proves the claim. Substituting this into (\ref{zy1}) we get
$$\pi^nz=\lpi(\pn](x))=\pi^n\lpi(x) \implies z=\lpi(x).$$
To prove the lemma for a general Lubin-Tate series $\px$, which may not be a polynomial, we can't directly use that $\bbc_K$ is algebraically closed to ensure that an element $y \in \dpi$ can be written as $\pn(x)$ for some $x \in \mm_{\bbc_K}$. Instead, we note that there is a unique power series $\theta(X) \in X O_K\llb X \rrb$ such that
\begin{equation}\label{cold1}
    [\pi](\theta(X))=\theta(X^q+\pi X)
\end{equation}
\cite[Corollary 2.16, p. 34]{milne}. We want to show for any $y \in \dpi$ there is an $x \in \mm_{\bbc_K}$ such that $y=\pn(x)$. After iterating, it is enough to show this for $n=1$. Fix an element $y \in \dpi$. Since $\bbc_K$ is algebraically closed, there is a $w \in \bbc_K$ such that
\begin{equation}\label{cold2}
    w^q+\pi w=\theta^{-1}(y).
\end{equation}
We claim that $w \in \mm_{\bbc_K}$, which follows from the fact that $\theta(X)$ is distance preserving on the maximal ideal; write $\displaystyle{\theta(X)=X+\sum_{i\geq 2}a_iX^i}$ for some $a_i \in O_K$. Then for any $r,s \in \mm_{\bbc_K}$ we have
\begin{equation}\label{cold3}
\theta(r)-\theta(s) = r-s+\sum_{i\geq 2}a_i(r-s)^i =(r-s)\lb 1+\sum_{i\geq 2}a_i\lb\sum_{j=0}^{n-1}r^is^{n-i-1} \rb \rb \in (r-s)O_K^\times, 
\end{equation}
and hence $v_K(\theta(r)-\theta(s))=v_K(r-s).$ Set $x:=\theta(w)$. Then $x \in \mm_{\bbc_K}$ by (\ref{cold3}) since $v_K(w)>0$ and moreover, we have $[\pi](x)=[\pi](\theta(w))=\theta(w^q+\pi w)=\theta(\theta^{-1}(y))=y$, by (\ref{cold1}) and (\ref{cold2}).
\end{proof}
 \begin{lemma}\label{genlemgen}
Let $K$ be a nonarchimedean local field of characteristic zero and $\px$ be a Lubin-Tate series for a chosen prime $\pi \in O_K$. For $x \in \mmc$ and $n \geq 1$ we have $\pn(x) \equiv x^{q^n}\bmod{\dpi}.$ 
\end{lemma}
\begin{proof}
We will use the congruence $\pn(X)\equiv X^{q^n}+\pi^n X\bmod{\pi X^2}$ for all $n\geq 1$. To show \linebreak $\pn(X)\equiv X^{q^n}+\pi^n X\bmod{\pi X^2}$ we induct on $n$. The base case follows from (\ref{ltseries}). Now suppose the congruence is true for some $n\geq 1$. Then
\begin{align}
 \label{hi}   [\pi^{n+1}](X)= [\pi^n]\lb\px\rb &\equiv (\px)^{q^n}+\pi^n(\px) \bmod{\pi X^2} \\
  \nonumber  &\equiv \lb \pi X+\cdots+X^q+\cdots\rb^{q^n}+\pi^n\lb \pi X+\cdots+X^q+\cdots\rb \bmod{\pi X^2} \\
 \nonumber   &\equiv X^{q^{n+1}}+\pi^{n+1}X\bmod{\pi X^2}. 
\end{align}
For each $x \in \mmc$ there is a $y \in O_{\bbc_K}$ such that
\begin{equation}\label{pnxshape}
    \pn(x)=x^{q^n}+\pi^nx+\pi x^2 y.
\end{equation}
Since both $v_K(\pi^nx)$ and $ v_K(\pi x^2 y)$ are at least $1$, both $\pi^nx, \pi x^2y \in \dpi$ by definition of $\dpi$.
\end{proof}
As a consequence of (\ref{pnxshape}) with $n=1$, if $L|K$ is a finite extension and $x \in \mml$ then
\begin{gather}\label{valpix}
    v_L([\pi](x))=\begin{cases}
        qv_L(x) \quad \quad \quad \quad \quad \text{ if } v_L(x^q)<v_L(\pi x), \text{ i.e., }v_L(x) <\frac{v_L(\pi)}{q-1}, \\
        v_L(\pi)+v_L(x) \quad  \quad \text{if } v_L(x^q)>v_L(\pi x), \text{ i.e., }v_L(x) >\frac{v_L(\pi)}{q-1}.
    \end{cases}
\end{gather}
The following is a result of Wiles \cite[Lemma 3]{wiles} (see also \cite[Proposition 2.2]{lubin94} and \cite[Lemma 1, p. 212]{lang}).
 \begin{theorem}\label{wiles}
Let $K$ be a nonarchimedean local field of characteristic zero and $\px$ be a Lubin-Tate series for a chosen prime $\pi \in O_K$. For all $x \in \mmc$,
$$\lpi(x)=\lim_{n\to \infty}\frac{[\pi^n](x)}{\pi^n}.$$
\end{theorem}
Equipped with Theorem \ref{wiles} we can prove that the torsion of the Lubin-Tate logarithm $\lpi(X)$ in $\mmc$ is precisely the set of all $[\pi]$-torsion in $\mmc$, analogous to the kernel of the usual $p$-adic logarithm $\log(1+X)$ being the set of all $p$-th power roots of unity.
\begin{lemma}\label{kernel} 
 Let $K$ be a nonarchimedean local field of characteristic zero and $\px$ be a Lubin-Tate series for a chosen prime $\pi \in O_K$. The kernel of $\lpi(X)$ in $\mmc$ is $\mup$.    
\end{lemma}
\begin{proof}
One containment is easy; if $\lambda \in \mup$ then there is a positive integer $N\geq 1$ such \linebreak that $\pn(\lambda)=0$ for all $n\geq N$, and hence $\displaystyle{\lim_{n\to\infty}\frac{\pn(\lambda)}{\pi^n}=0}$. 
For the reverse containment, \linebreak suppose $\lambda \in \mmp$ and $\lpi(\lambda)=0$. Multiplying through by any power of $\pi$ tells us that \linebreak
$\lpi(\pn(\lambda))=0$ for all $ n \geq 1$.
By Lemma \ref{genlemgen}, we can choose a positive integer $n\geq 1$ sufficiently large so that $\pn(\lambda) \in \dpi$. It now follows from injectivity of $\lpi$ on $\dpi$ that $\pn(\lambda)=0$ and\linebreak hence $\lambda \in \mupn \subset \mup$.
\end{proof}
We conclude with two theorems. The first is an analogue of the decomposition of principal units 
$$1+\mml \cong \mu_{p^\infty}(L) \times \zp^{[L:K]}$$
for a finite extension $L|\qp$, where $\mu_{p^\infty}(L)$ is the set of $p$-power roots of unity in $L$.
\begin{theorem}\label{decomplemma}
Let $L|K$ be a finite extension. There is an isomorphism of $O_K$-modules
$$\ltml \cong \mup(L) \times O_K^{[L:K]}.$$
\end{theorem}
\begin{proof}
For $i> \displaystyle{\frac{v_L(\pi)}{q-1}}$, $\lpi: \ltmli \longrightarrow \mml^i$ is an isomorphism of $O_K$-modules. 
As additive $O_K$-modules, 
$$\mml^i \cong O_L \cong O_K^{[L:K]}. $$
Hence for $i>\displaystyle{\frac{v_L(\pi)}{q-1}}$ (i.e., $\mml^i \subseteq \dpil$), $\lpi(X)$ induces an isomorphism of $O_K$-modules 
\begin{equation}\label{freeiso}
    \ltmli \stackrel{\cong}{\longrightarrow} O_K^{[L:K]}.
\end{equation}
Since $\ltmli/\ltmlii \cong k_L$ for all $i\geq 1$ by (\ref{isotores}), $[\ltml:\ltmli]=q^{f(i-1)}$, where $f$ is the residue field degree. In particular $\ltmli$ has finite index in $\ltml$ and it follows that $\ltml$ is a finitely generated $O_K$-module with the same rank $[L:K]$. To see why, note that $\ltml/\ltmli$ is a torsion $O_K$-module. In particular, 
$$\ltml/\ltmli \otimes_{O_K} K=0,$$
since $K$ is an injective $O_K$-module. Since $K$ is a flat $O_K$-module, tensoring the short exact sequence 
\begin{equation}\label{ses}
    0 \longrightarrow \ltmli \longrightarrow \ltml \longrightarrow \ltml/\ltmli \longrightarrow 0
\end{equation}
by $K$ shows that $$\ltmli \otimes_{O_K} K \cong \ltml \otimes_{O_K} K.$$ Since the rank of $\ltml$ as an $O_K$-module is exactly the dimension of $\ltml \otimes_{O_K} K$ as a $K$-vector space, this proves that the rank of $\ltml$ is equal to the rank of $\ltmli$, which is $[L:K]$ by (\ref{freeiso}). The fact that $\ltml$ is finitely generated follows from (\ref{ses}) since both $\ltmli$ and $\ltml/\ltmli$ are finitely generated.
Since $\ltml$ is a finitely generated $O_K$-module with rank $[L:K]$ it follows by the structure theorem for finitely generated modules over a PID that
$\ltml \cong T \times O_K^{[L:K]}$
for some finite torsion submodule $T$. The torsion in $\ltml$ is exactly the set of all $[\pi]$-power torsion in $L$, namely $\mup(L)$, which is finite, hence the result follows.
\end{proof}
\begin{theorem}\label{freelem} For any finite extension $L|K$, $\lpi(\ltml)$ is a free $O_K$-module of rank $[L:K]$.
\end{theorem}
\begin{proof}
Using an analogous argument to the one seen in Theorem \ref{decomplemma}, one can show that $\lpi(\ltml)$ is a finitely generated $O_K$-module of rank $[L:K]$ so by the structure theorem for finitely generated modules over a PID, 
$$\lpi(\ltml) \cong T \times O_K^{[L:K]}$$ for some torsion submodule $T$. Since $\lpi(\ltml) \subseteq L$ has the usual $O_K$-module structure, $\lpi(\ltml)$ is torsion-free, thus $T=0$ and $\lpi(\ltml) \cong O_K^{[L:K]}$.
\end{proof}
\section{$\pi$-Regular extensions}\label{pireg}

In \cite[pp. 231-233]{cohen}, Cohen considers finite extensions $L|\qp$ that do not contain any nontrivial $p$-power root of unity and uses Nakayama's lemma to determine a $\zp$-basis of the principal units $1+\mml$, as follows \cite[Corollary 4.3.19]{cohen}.
\begin{theorem}\label{cohen}
    Let $L|\qp$ be a finite extension such that $L$ does not contain a non-trivial $p$-th root of unity, and denote by $e$ and $f$ the ramification index and residue field degree of $L|\qp$. Let $\varpi$ be a uniformiser of $L$ and let $\zeta_1,\ldots,\zeta_f$ in $O_L$ be a lift of an $\fp$-basis of the residue field of $L$. The $ef$ elements
$$1+\zeta_i\pi^j \text{ with } 1\leq i \leq f, \ 1\leq j \leq \frac{pe}{p-1} \text{ and } p\nmid j$$
constitute a $\zp$-basis of the multiplicative $\zp$-module $1+\mml$. 
\end{theorem}
In this section we will follow a similar route to determine an $O_K$-basis for the $O_K$-module $\ltml$.
\begin{definition}
    Call a finite extension $L|K$ \textit{$\pi$-regular} if $L$ contains no nontrivial $[\pi]$-torsion: $\mup(L)=\{0\}.$
\end{definition}

\begin{lemma}\label{equivdefofreg}
    Let $L|K$ be a finite extension and $\varpi \in O_L$ be a prime in $L$. The following are equivalent statements.
    \begin{itemize}
        \item[(i)] $L|K$ is $\pi$-regular, i.e. $\mup(L)=\{0\}$,
        \item[(ii)] Either $\displaystyle{\frac{v_L(\pi)}{q-1} \not\in \z}$, or $\displaystyle{\frac{v_L(\pi)}{q-1} \in \z}$ and there is no $u \in O_L^\times$ satisfying
        \begin{equation}\label{reg1}
            u^q+\varepsilon u\equiv 0\bmod{\mml}, 
        \end{equation}
        where $\varepsilon \in O_L^\times$ comes from the decomposition $\pi=\varepsilon\varpi^{v_L(\pi)}$.
    \end{itemize}
\end{lemma}
\begin{proof}
For the implication (i)$\implies$(ii), assume that $L|K$ is $\pi$-regular so that $\mup(L)=\{0\}$. If  $\displaystyle{\frac{v_L(\pi)}{q-1} \not\in \z}$ then we're already done, so suppose $\displaystyle{\frac{v_L(\pi)}{q-1} \in \z}$ and assume there is a $u \in O_L^\times$ satisfying the congruence (\ref{reg1}) for $\varepsilon:=\displaystyle{\frac{\pi}{\varpi^{v_L(\pi)}}}$. Rearranging (\ref{reg1}) shows $-\varepsilon$ is a $(q-1)$-th power in the residue field $k_L$. By Hensel's lemma, there is $\xi \in O_L^\times$ such that $-\varepsilon=\xi^{q-1}$, so
\begin{align*}
  -\frac{\pi}{\varpi^{v_L(\pi)}}=\xi^{q-1} &\implies -\pi=\lb \xi\varpi^{v_L(\pi)/(q-1)}\rb^{q-1}\\
    &\implies \lb \xi\varpi^{v_L(\pi)/(q-1)}\rb^{q}+\pi\lb \xi\varpi^{v_L(\pi)/(q-1)}\rb=0.
\end{align*}
In particular, if $\px=X^q+\pi X$ is the basic Lubin-Tate series for $\pi$ then we have found nonzero $\xi\varpi^{v_L(\pi)/(q-1)} \in O_L$ satisfying $[\pi](\xi\varpi^{v_L(\pi)/(q-1)})=0$, contradicting our assumption that $L|K$ is \linebreak $\pi$-regular. In the case that $\px$ is not the basic Lubin-Tate series for $\pi$, conjugation on $\mml$ by an isomorphism of formal group laws defined over $O_K$ \cite[Corollary 2.16, p. 34]{milne} shows that any nonzero solution to $X^q+\pi X=0$ in $\mml$ will give rise to a nonzero solution of $\px=0$. Thus we have proved there can be no $u \in O_L^\times$ satisfying the congruence (\ref{reg1}) when $L$ contains no nontrivial $[\pi]$-torsion.

For (ii)$\implies$(i), if $\displaystyle{\frac{v_L(\pi)}{q-1} \not\in \z}$ then $L|K$ is $\pi$-regular since $K_\pi|K$ is totally ramified of degree $q-1$. If $\displaystyle{\frac{v_L(\pi)}{q-1} \in \z}$ and $K_\pi \subseteq L$, then we will find $u \in O_L^\times$ satisfying (\ref{reg1}). This shows that if there is no $u \in O_L^\times$ satisfying (\ref{reg1}) then $L$ does not contain $K_\pi$, so $L$ is $\pi$-regular.

If $K_\pi \subseteq L$ then we can pick a $\lambda \in \mml$ that generates the cyclic $O_K$-module $\mu_{[\pi]}$. By (\ref{pnxshape}) with $n=1$, any $x \in \mml$ satisfies \begin{equation}\label{reg2}
    [\pi](x) = x^q+\pi x+\pi x^2 y
\end{equation}
for some $y \in O_L$. 
Since $v_L(\lambda)$ is $i:=\displaystyle{\frac{v_L(\pi)}{q-1}} \geq 1$, we can write $\lambda=u\varpi^i$ for some $u \in O_L^\times$. Since $[\pi](\lambda)=0$, substituting $u\varpi^i$ for $\lambda$ in (\ref{reg2}) implies
\begin{align*}
     (u\varpi^i)^q+\pi u\varpi^i+\pi (u\varpi^i)^2 y=0 &\implies u^q+\varepsilon u+\varepsilon u^2\varpi^i y =0 \\
     &\implies u^q+\varepsilon u \equiv 0\bmod{\varpi O_L},
\end{align*}
where the first implication follows by dividing through by $\varpi^{qi}$. This completes the proof.
\end{proof}
\begin{remark}
    We note that (ii) in Lemma \ref{equivdefofreg} is independent of the choice of prime in $O_L$ that defines $\varepsilon$. Suppose $\displaystyle{\frac{v_L(\pi)}{q-1} \in \z}$. Let $\varpi_1,\varpi_2 \in O_L$ be prime and $\varepsilon_1, \varepsilon_2 \in O_L^\times$ such that $$\varepsilon_1\varpi_1^{v_L(\pi)}=\pi=\varepsilon_2\varpi_2^{v_L(\pi)}.$$
    If $u \in O_L^\times$ satisfies 
    $u^q+\varepsilon_1u\equiv 0\bmod{\mml},$ then
    $v:=\displaystyle{u\lb \frac{\varpi_1}{\varpi_2}\rb^{\frac{v_L(\pi)}{q-1}}}$ satisfies
    $v^q+\varepsilon_2v\equiv0\bmod{\mml}.$
\end{remark}

For a finite extension $L|\qp$, the $p$-th power map induces a group homomorphism $$(1+\mml^i)/(1+\mml^{i+1}) \stackrel{\cong}{\longrightarrow} (1+\mml^{pi})/(1+\mml^{pi+1})$$ whenever $1\leq i\leq\displaystyle{\frac{v_L(p)}{p-1}}$. This induced map on quotients can be viewed as a homomorphism of $\zp$-modules where $\zp$ acts by exponentiation. Moreover, this map is an isomorphism of $\zp$-modules whenever $i<\displaystyle{\frac{v_L(p)}{p-1}}$, or $i=\displaystyle{\frac{v_L(p)}{p-1}}$ and $L$ does not contain a nontrivial $p$-th root of unity \cite[Proposition (1) and (2), p. 15]{iso}.
The following two theorems extend this to the setting of Lubin-Tate formal group laws.
\begin{theorem}\label{FV}
    When $1\leq i \leq \displaystyle{\frac{v_L(\pi)}{q-1}},$ the series $[\pi](X)$ induces an $O_K$-module homomorphism
    \begin{equation}\label{iso}
        \ltmli/\ltmlii \longrightarrow \ltmlqi/\ltmlqii,
    \end{equation}
    and it is an $O_K$-module isomorphism if $i <\displaystyle{\frac{v_L(\pi)}{q-1}}.$
\end{theorem}
\begin{proof}
First we show $[\pi](\ltmli) \subseteq \ltmlqi$.
For $x\in \ltmli$, write 
$  [\pi](x) = x^q+\pi x+\pi x^2 y,$
for some $y \in O_L$. To show $[\pi](x) \in \ltmlqi$, the case $x=0$ is clear, so let $x \neq 0$. Since $v_L(x) \geq i$,
\begin{equation*}\label{fri2}
    v_L([\pi](x))\geq \min(qv_L(x), v_L(\pi)+v_L(x)) \geq \min(qi, v_L(\pi)+i).
\end{equation*}
When $1 \leq i \leq \displaystyle{\frac{v_L(\pi)}{q-1}}$, $\min(qi,v_L(\pi)+i)=qi$ so $v_L([\pi](x)) \geq qi$ hence $[\pi](\ltmli) \subseteq \ltmlqi$. Similarly, $[\pi](\ltmlii)\subseteq \ltmlqii$ so the induced map (\ref{iso}) on quotients is well defined. Moreover, (\ref{iso}) is $O_K$-linear due to $[a] \circ [\pi]=[\pi]\circ [a]$ in $O_K\llb X \rrb $ for all $a \in O_K$. The map (\ref{iso}) is injective if $1\leq i<\displaystyle{\frac{v_L(\pi)}{q-1}}$: if $v_L(x)=i$ then $$v_L([\pi](x))=\min(qi,v_L(\pi)+i)=qi,$$ so $[\pi](x) \not\in \ltmlqii$. Surjectivity now follows as both quotients in (\ref{iso}) are finite with order $q^f$ by (\ref{isotores}), where $f$ is the residue field degree of $L|K$. We conclude that for $1\leq i < \displaystyle{\frac{v_L(\pi)}{q-1}}$, the Lubin-Tate series $\px$ induces an isomorphism of $O_K$-modules $\ltmli/\ltmlii \longrightarrow  \ltmlqi/\ltmlqii.$
\end{proof}
\begin{theorem}\label{regisolem}
   Let $L|K$ be a finite extension with ramification index divisible by $q-1$. \linebreak Fix $i:=\displaystyle{\frac{v_L(\pi)}{q-1}} \in \z^+$. The series $\px$ induces an isomorphism of $O_K$-modules
  \begin{equation}\label{iso*}
        \ltmli/\ltmlii \longrightarrow  \ltmlqi/\ltmlqii
    \end{equation}
    if and only if $L|K$ is $\pi$-regular.
\end{theorem}
\begin{proof}
By Theorem \ref{FV}, (\ref{iso*}) is a well-defined homomorphism of $O_K$-modules. Since both the domain and codomain in (\ref{iso*}) have order $q^f$ where $f$ is the residue field degree of $L|K$, it suffices to prove that (\ref{iso*}) is injective if and only if $L|K$ is $\pi$-regular.

The map (\ref{iso*}) fails to be injective if and only if there is an $x \in \mml$ such that $v_L(x)=i$ and $v_L([\pi](x))\geq qi+1$. As in the proof of Theorem \ref{FV}, write
\begin{equation}\label{regiso1}
    [\pi](x) = x^q+\pi x+\pi x^2 y
\end{equation}
for some $y \in O_L$. Since $v_L(x)=i$, there is some $u \in O_L^\times$ such that $x=u\varpi^i$. Substitute $x=u\varpi^i$ and $\pi=\varepsilon \varpi^{v_L(\pi)}$ into (\ref{regiso1}) to get
\begin{equation}
    [\pi](x) =u^q\varpi^{qi}+\varepsilon u\varpi^{i+v_L(\pi)}+\varepsilon u^2 \varpi^{2i+v_L(\pi)}y.
\end{equation}
By the definition of $i$ we have $qi=i+v_L(\pi)$, so
$$[\pi](x) \equiv \varpi^{qi}(u^q+\varepsilon u)\bmod{\mml^{qi+1}}.$$
In particular, since $i=\displaystyle{\frac{v_L(\pi)}{q-1} \in \z^+}$, 
\begin{align*}
    v_L([\pi](x)) \geq qi+1 &\iff u^q+\varepsilon u\equiv 0\bmod{\mml} \iff L|K \text{ fails to be } \pi\text{-regular}
\end{align*}
by Lemma \ref{equivdefofreg}.
Hence (\ref{iso*}) is injective if and only if $L|K$ is $\pi$-regular. 
\end{proof}
\begin{remark}\label{rmk1}
    As a consequence of Theorems \ref{FV} and \ref{regisolem}, the series $\pn(X)$ induces isomorphisms 
$$\ltmli/\ltmlii \stackrel{\cong}{\longrightarrow} \mathcal{F}(\mml^{q^ni})/\mathcal{F}(\mml^{q^ni+1}) \ \text{ if } \displaystyle{1\leq i <\frac{v_L(\pi)}{q^{n-1}(q-1)}},$$
and an isomorphism
$$\ltmli/\ltmlii \stackrel{\cong}{\longrightarrow} \mathcal{F}(\mml^{q^ni})/\mathcal{F}(\mml^{q^ni+1}) \text{ for } i =\frac{v_L(\pi)}{q^{n-1}(q-1)} \iff L|K \text{ is } \pi\text{-regular}.$$
\end{remark}

The following theorem extends part (3) of the Proposition in \cite[p. 15]{iso} which says for a finite extension $L|\qp$ with ramification index $e$, the $p$-th power map induces an isomorphism of $\zp$-modules
$$(1+\mml^i) /(1+\mml^{i+1})\longrightarrow (1+\mml^{i+e})/(1+\mml^{i+e+1})\ \text{ when } i> \displaystyle{\frac{e}{p-1}}.$$ 
\begin{theorem}\label{FV3}
  When $i > \displaystyle{\frac{v_L(\pi)}{q-1}}$, the series $[\pi](X)$ induces an $O_K$-module isomorphism
  \begin{equation}\label{iso**}
      \ltmli/\ltmlii \longrightarrow \mathcal{F}(\mml^{i+v_L(\pi)})/\mathcal{F}(\mml^{i+v_L(\pi)+1})
  \end{equation}
\end{theorem}
\begin{proof}
This map is well defined by (\ref{valpix}): if $v_L(x)>\displaystyle{\frac{v_L(\pi)}{q-1}}$ then $v_L([\pi](x))=v_L(\pi)+v_L(x)$ so
$$[\pi](\ltmli) \subseteq \mathcal{F}(\mml^{i+v_L(\pi)})$$ for all $i> \displaystyle{\frac{v_L(\pi)}{q-1}}$. The map (\ref{iso**}) is $O_K$-linear by the way $\ltmli$ is an $O_K$-module and it is injective by (\ref{valpix}) since if $v_L(x)=i$ then $v_L([\pi](x))=i+v_L(\pi)$, so $[\pi](x)$ is nonzero in $ \mathcal{F}(\mml^{i+v_L(\pi)})/\mathcal{F}(\mml^{i+v_L(\pi)+1})$. Finally, since $$|\ltmli/\ltmlii |=q^f=| \mathcal{F}(\mml^{i+v_L(\pi)})/\mathcal{F}(\mml^{i+v_L(\pi)+1})|,$$ the map (\ref{iso**}) is surjective.
\end{proof}
By Theorem \ref{decomplemma}, $\ltml$ is a finitely generated $O_K$-module, so by Nakayama's lemma a set $\mathcal{B} \subseteq \ltml$ generates $\ltml$ as an $O_K$-module if and only if the image of $\mathcal{B}$ in $\ltml/[\pi](\ltml)$ generates it as a $k$-vector space.

\begin{theorem}\label{basisthm1}
 For a $\pi$-regular extension $L|K$ with uniformiser $\varpi$ and residue field degree $f$, the set of $[L:K]$ elements 
 \begin{equation*}
     \mathcal{B}_L:=\left\{\zeta_i \varpi^j \mid 1\leq i \leq f, \, 1\leq j \leq \frac{qv_L(\pi)}{q-1}, \, q \nmid j \right \}
 \end{equation*}
 is an $O_K$-basis of $\ltml$, where the $\zeta_i$'s in $O_L$ are any lift of a $k$-basis for $k_L$.
\end{theorem}
\begin{proof}
  Pick $\zeta_1,\ldots,\zeta_f \in O_L$ to be a lift of a $k$-basis for $k_L$ and fix $j \geq 1$. Let $\mathcal{C}\subseteq O_K$ be a set of coset representatives for the residue field of $K$. We start by defining a \textit{generating set of level $j$}. Note that any coset representative $x$ of $\ltmlj/\ltmljj$ can be written uniquely in the form
  \begin{equation}\label{cosetrep}
      x=\Fsum_{i=1}^f\,\,\,\,\,\,[c_i](\zeta_i\varpi^j), \text{ for some } c_i \in \mathcal{C},
  \end{equation}
  where $\sum_F$ denotes repeated addition with respect to the Lubin-Tate module structure on $\ltml$. To see why, note that
  $$\Fsum_{i=1}^f\,\,\,\,\,\,[c_i](\zeta_i\varpi^j) \equiv \sum_{i=1}^f c_i\zeta_i\varpi^j\bmod{\mml^{j+1}}$$
  In particular, 
  $\displaystyle{\Fsum_{i=1}^f\,\,\,\,\,\,[c_i](\zeta_i\varpi^j)}$ is mapped to $\displaystyle{\sum_{i=1}^f c_i\zeta_i}$ in $k_L$ by (\ref{isotores}). Since any element of $k_L$ can be written in this form, pulling back via the isomorphisms
$$\ltmlj/\ltmljj \stackrel{\cong}{\longrightarrow}k_L$$
  (as seen in (\ref{isotores})) proves the claim. Since any coset representative $x$ of $\ltmlj/\ltmljj$ can be written as in (\ref{cosetrep}), the set 
  $$B_j:=\{\zeta_i\varpi^j\mid 1\leq i \leq f\}$$
  generates $\ltmlj/\ltmljj$ as an $O_K$-module with its Lubin-Tate structure. Following the approach in \cite[Corollary 4.3.19]{cohen}, we call $B_j$ a \textit{generating set of level $j$} and we say that the collection of these generating sets 
$$\displaystyle{\bigcup_{j\geq 1} B_j}$$
forms a \textit{generating system} of $\ltml$.
By Theorem \ref{FV3}, $\px$ induces isomorphisms 
$$\mathcal{F}(\mml^j)/\mathcal{F}(\mml^{j+1}) \longrightarrow \mathcal{F}(\mml^{j+v_L(\pi)})/\mathcal{F}(\mml^{j+v_L(\pi)+1})$$
when $j > \displaystyle{\frac{v_L(\pi)}{q-1}} $, so $[\pi](B_j)$ is a generating set of level $j+v_L(\pi)$ whenever $j > \displaystyle{\frac{v_L(\pi)}{q-1}} $. In particular, it follows by Nakayama's lemma that the generating set of level $j$ where $j>\displaystyle{\frac{qv_L(\pi)}{q-1}}$
can be suppressed from our generating system of $\ltml$ and what remains still generates $\ltml$ as an $O_K$-module.  

Moreover, if $1\leq j \leq \displaystyle{\frac{v_L(\pi)}{q-1}}$, $[\pi](B_j)$ is a generating set of level $qj$ by Theorems \ref{FV} and \ref{regisolem} since $\px$ induces isomorphisms of $O_K$-modules
$$\mathcal{F}(\mml^j)/\mathcal{F}(\mml^{j+1}) \longrightarrow \mathcal{F}(\mml^{qj})/\mathcal{F}(\mml^{qj+1})$$
when $1\leq j \leq \displaystyle{\frac{v_L(\pi)}{q-1}}$. In particular, it follows again by Nakayama's lemma that the levels $j$ divisible by $q$ can be suppressed from our generating system of $\ltml$.
We are left with the generating system
\begin{equation*}
     \mathcal{B}_L:=\left\{\zeta_i \varpi^j \mid 1\leq i \leq f, \, 1\leq j \leq \frac{qv_L(\pi)}{q-1}, \, q \nmid j \right \}, 
 \end{equation*}
 of $\ltml$, where
 \begin{equation}\label{sizeofbl}
     |\mathcal{B}_L|=f\lb \left\lfloor \frac{qv_L(\pi)}{q-1}\right\rfloor -  \left\lfloor \frac{v_L(\pi)}{q-1}\right\rfloor \rb =fv_L(\pi)=[L:K].
 \end{equation}
Since $L|K$ is $\pi$-regular, $\mup(L)=\{0\}$ and thus it follows by Theorem \ref{decomplemma} that $\ltml$ is a free $O_K$-module of rank $[L:K]$ and thus (\ref{sizeofbl}) implies that the generating system $\mathcal{B}_L$ is an $O_K$-basis of $\ltml$.
\end{proof}
\begin{remark}\label{rmk410}
 In the proof of Theorem \ref{basisthm1} we used $\pi$-regularity only once; to remove the level $j=\displaystyle{\frac{qv_L(\pi)}{q-1}}$ from our generating system of $\ltml$ if $\displaystyle{\frac{v_L(\pi)}{q-1}} \in \z$. To suppress $j=\displaystyle{\frac{qv_L(\pi)}{q-1}}$ when $\displaystyle{\frac{v_L(\pi)}{q-1}} \in \z$ we used Nakayama's lemma and that $\px$ induces an $O_K$-module homomorphism
\begin{equation}\label{rmkeq}
  \ltmlj/\ltmljj \longrightarrow\ltmlqj/\ltmlqjj  \text{ for }j=\displaystyle{\frac{qv_L(\pi)}{q-1}},
\end{equation}
 that is an isomorphism of $O_K$-modules if $L|K$ is $\pi$-regular, by Theorem \ref{regisolem}. When we introduce nontrivial $[\pi]$-torsion later in Section \ref{extension}, we will have to add back the level $j=\displaystyle{\frac{qv_L(\pi)}{q-1}}$ (which is an integer since $K_\pi \subseteq L \implies q-1\mid v_L(\pi)$) into our generating system of $\ltml$. What remains will be a spanning set of $\ltml$ of size $[L:K]+f$ (this is Theorem \ref{genspanthm}).
\end{remark}
\begin{remark}
Theorem \ref{basisthm1} includes the classical multiplicative case of $1+\mml$ as a $\zp$-module, as given in Theorem \ref{cohen} when $\mu_{p^\infty}(L)=\{1\}$. Take $K=\qp$, $[p](X)=(1+X)^p-1$, and a shift by one to switch the domain from the maximal ideal $\mml$ to the principal units $1+\mml$. The analogue of (\ref{cosetrep}) in the multiplicative setting says that any coset representative $x$ of $(1+\mml^j)/(1+\mml^{j+1})$ can be written uniquely in the form 
$$\prod_{i=1}^f(1+\zeta_i\varpi^j)^{c_i} \ \text{ for some } c_i \in \{0,\ldots,p-1\},$$
where the $\zeta_i$'s are a lift of an $\fp$-basis of $\fpf$. To see why, note that
$$\prod_{i=1}^f(1+\zeta_i\varpi^j)^{c_i} \equiv 1+\sum_{i=1}^fc_i\zeta_i\varpi^j \bmod{\mml^{j+1}}.$$
Hence, under the isomorphism $(1+\mml^j)/(1+\mml^{j+1}) \stackrel{\cong}{\longrightarrow} \fpf$, this product is mapped to $\displaystyle{\sum_{i=1}^fc_i\zeta_i}$.
\end{remark}
\section{A lower bound for the valuations in $\lpi(\ltml)$ for $\pi$-regular extensions}\label{main}
In this section we compute a basis for the $O_K$-module $\lpi(\ltml)$ when $L$ is a $\pi$-regular extension of $K$ and use it to determine an element of lowest valuation in $\lpi(\ltml)$. Recall that in this setting, 
\begin{equation*}
     \mathcal{B}_L:=\left\{\zeta_i \varpi^j \mid 1\leq i \leq f, \, 1\leq j \leq \frac{qv_L(\pi)}{q-1}, \, q \nmid j \right \}
 \end{equation*}
 is a basis for the $O_K$-module $\ltml$. Hence the set
 \begin{equation}\label{defofal}
     \widetilde{\mathcal{B}}_L:=\left\{\lpi(\zeta_i \varpi^j) \mid 1\leq i \leq f, \, 1\leq j \leq \frac{qv_L(\pi)}{q-1}, \, q \nmid j \right \}, 
 \end{equation}
 spans the additive group $\lpi(\mml)$ as an $O_K$-module.
 \begin{theorem}\label{basisthm2}
     For a $\pi$-regular extension $L|K$ with uniformiser $\varpi$ and residue field degree $f$, the set of $[L:K]$ elements 
    $$  \widetilde{\mathcal{B}}_L:=\left\{\lpi(\zeta_i \varpi^j) \mid 1\leq i \leq f, \, 1\leq j \leq \frac{qv_L(\pi)}{q-1}, \, q \nmid j \right \}$$
is a basis of the additive group $\lpi(\ltml)$ as an $O_K$-module, where the $\zeta_i$'s in $O_L$ are any lift of a $k$-basis for $k_L$.
 \end{theorem}
\begin{proof}
Since $L|K$ is $\pi$-regular, $\mup(L)=\{0\}$ so $\lpi:\ltml \longrightarrow \lpi(\ltml)$ is an $O_K$-module isomorphism. Hence the $O_K$-basis $\mathcal{B}_L$ of $\ltml$ is mapped to an $O_K$-basis $\widetilde{B}_L$ of $\lpi(\ltml)$.
\end{proof}
Since the $\lpi(\zeta_i \varpi^j)$'s span the additive group $\lpi(\ltml)$ as an $O_K$-module, the valuations of the elements in $\lpi(\ltml)$ are bounded below by the valuations of the elements of $\widetilde{\mathcal{B}}_L$. The following theorem gives an explicit formula for the valuation of elements in $\lpi(\ltml)$.

\begin{theorem}\label{genlogvallem}
Suppose $L|K$ is $\pi$-regular, $x \in \mml$, and $\ell_x$ is the smallest nonnegative integer satisfying $[\pi^{\ell_x}](x) \in \dpil$. Then 
\begin{equation}\label{logval}
    v_L(\lpi(x)) = q^{\ell_x} v_L(x)-{\ell_x} v_L(\pi).
\end{equation}
Moreover,
\begin{gather}\label{ellxeq}
    \ell_x= \begin{cases}
        0 \quad \quad \quad \quad \quad \quad \,\,\quad \quad \quad \,\,\,\, \text{ if } x \in \dpil, \text{ i.e., } v_L(x)>\frac{v_L(\pi)}{q-1}, \\
        \llf\log_q\lb \frac{v_L(\pi)}{v_L(x)(q-1)}\rb \rrf +1\,\,\,\, \text{ otherwise},
    \end{cases}
\end{gather}
where $\log_q$ denotes the base $q$ logarithm in $\real$. 
\end{theorem}
\begin{proof}
When $x=0$, the statement of the theorem is true since $\ell_0=0$ and $\lpi(0)=0$, so both sides of (\ref{logval}) are infinite. When $x\neq 0$, we'll consider the two cases $x \in \dpil$ and $x \not\in\dpil$ separately. 

For the first case, if $x \in \dpil$ then $\ell_x=0$. Since $\lpi$ is an isometry $\mathcal{F}(\dpil) \longrightarrow \dpil$, we have $v_L(\lpi(x))=v_L(x)$, which is exactly (\ref{logval}) since $\ell_x=0$.

Now consider the case $x \not \in \dpil$ so $v_L(x)\leq \displaystyle{\frac{v_L(\pi)}{q-1}}$. We will first prove (\ref{logval}), and then show 
\begin{equation}\label{logval.1}
   \ell_x= \llf\log_q\lb \frac{v_L(\pi)}{v_L(x)(q-1)}\rb \rrf +1 .
\end{equation}
Since $[\pi^{\ell_x}](x) \in\dpil$ we have
\begin{equation}\label{logval.2}
    v_L(\lpi([\pi^{\ell_x}](x)))=v_L([\pi^{\ell_x}](x)). 
\end{equation}
On the other hand, 
\begin{equation}\label{logval.3}
    v_L(\lpi([\pi^{\ell_x}](x)))=v_L(\pi^{\ell_x}\lpi(x))=\ell_xv_L(\pi)+v_L(\lpi(x)).
\end{equation}
Hence (\ref{logval.2}) and (\ref{logval.3}) imply 
\begin{equation}\label{logval.4}
    v_L(\lpi(x))=v_L([\pi^{\ell_x}](x))- \ell_xv_L(\pi).
\end{equation}
To prove (\ref{logval}), it remains to show that $v_L([\pi^{\ell_x}](x))=q^{\ell_x}v_L(x)$, which follows by Remark \ref{rmk1} if
\begin{equation}\label{logval.5}
    v_L(x) \leq \frac{v_L(\pi)}{q^{\ell_x-1}(q-1)}.
\end{equation}
To prove the claim (\ref{logval.5}) for $x \not\in\dpil$, suppose to the contrary that
$$ v_L(x) >\frac{v_L(\pi)}{q^{\ell_x-1}(q-1)}.$$
Then, 
$$v_L(x^{q^{\ell_x-1}})=q^{\ell_x-1}v_L(x) >\frac{v_L(\pi)}{q-1} \implies x^{q^{\ell_x-1}} \in \dpil$$
by the definition of $\dpil$. But $x^{q^{\ell_x-1}}\equiv [\pi^{\ell_x-1}](x)\bmod\dpil$ by Lemma \ref{genlemgen} so \linebreak $[\pi^{\ell_x-1}](x) \in \dpil$, contradicting the minimality of $\ell_x$. Thus we have proved (\ref{logval.5}) for $x \notin \dpil$, and in particular it follows that $v_L([\pi^{\ell_x}](x))=q^{\ell_x}v_L(x)$. Substitute this into (\ref{logval.4}) to see that
$$v_L(\lpi(x))=q^{\ell_x}v_L(x)-\ell_xv_L(\pi),$$
completing the proof of (\ref{logval}). To prove (\ref{logval.1}), for any $n\geq 1$ we have
$$\pn(x) \equiv x^{q^n}\bmod\dpil$$
by Lemma \ref{genlemgen}, so $\pn(x) \in \dpil$ if and only if $x^{q^n} \in \dpil$. By the definition of $\dpil$,
\begin{equation*}\label{logval.6}
     x^{q^n} \in \dpil \iff q^nv_L(x) > \frac{v_L(\pi)}{q-1} \iff n>\log_q\lb \frac{v_L(\pi)}{v_L(x)(q-1)}  \rb.
\end{equation*}
Thus the minimal $n$ for which $[\pi^n](x) \in \dpil$ when $x\notin \dpil$ is 
\[
 \ell_x= \llf\log_q\lb \frac{v_L(\pi)}{v_L(x)(q-1)}\rb \rrf +1 . \qedhere
\]

\end{proof}
\begin{remark}
    In the multiplicative setting ($K=\qp$ and $[p](X)=(1+X)^p-1$), Theorem \ref{genlogvallem} can be restated in the following way. Let $L$ be a finite extension of $\qp$ that does not contain a primitive $p$-th root of unity. Let $x \in 1+\mml$ and $\ell_x$ be the smallest nonnegative integer satisfying $x^{\ell_x} \in 1+\mathcal{D}_{\qp}(L)$. Then $v_L(\log(x))=p^{\ell_x}v_L(x)-\ell_x v_L(p)$.
\end{remark}
\begin{remark}
   The description of $v_L(\lpi(x))$ in Theorem \ref{genlogvallem} also appears in \cite[Proposition 6.4]{siggymoore}, where the proof uses Theorem \ref{wiles} to deduce (\ref{logval}), rather than taking cases on whether or not $x$ is in $\dpil$. 
\end{remark}
\begin{cor}\label{dontcareaboutunit}
    Suppose $L|K$ is $\pi$-regular. For any $x \in \mml$ and $u \in O_L^\times$ we have
    $$v_L(\lpi(u x))=v_L(\lpi(x)).$$
\end{cor}
\begin{proof}
    If we can show $\ell_{u x}=\ell_x$ in Theorem \ref{genlogvallem} then we're done since $v_L(u)=0$ implies that
    $$q^{\ell_x}v_L(u x)-\ell_xv_L(\pi)=q^{\ell_x}v_L(x)-\ell_xv_L(\pi).$$
    If $x\in \dpil$ then $ux\in \dpil$ and both $\ell_x$ and $\ell_{ux}$ are $0$, so suppose not. To prove that $\ell_{u x}=\ell_x$ we will show for $\ell \geq 1$ that $[\pi^\ell](x) \in \dpil \iff [\pi^\ell](u x) \in \dpil.$ This follows by Lemma \ref{genlemgen}:
    \begin{align*}
         [\pi^\ell](u x) \in \dpil &\iff (ux)^{q^\ell} \in \dpil \iff x^{q^\ell} \in \dpil \iff [\pi^\ell](x) \in \dpil. \qedhere
    \end{align*}
\end{proof}
Note that when the ramification index $v_L(\pi)$ of $L|K$ is less than $q-1$,
$$\frac{v_L(\pi)}{q-1}<1 \implies \dpil=\mml \implies \lpi(\ltml)=\mml,$$
so $v_L(y)\geq 1$ for all $y \in \lpi(\ltml)$. In the case that $v_L(\pi)\geq q-1$ we now compute the minimal valuation of the elements in $\lpi(\ltml)$ when $L|K$ is $\pi$-regular by using Theorem \ref{genlogvallem}.
\begin{theorem}\label{minvalthm}
    Let $L|K$ be $\pi$-regular and suppose $\displaystyle{\frac{v_L(\pi)}{q-1}}\geq 1$, so 
\begin{equation}\label{g.1}
    q^{\gamma-1}\leq \frac{v_L(\pi)}{q-1}<q^\gamma
\end{equation}
    for a unique integer $\gamma \geq 1$. Then, for any prime $\varpi$ of $L$,
    \begin{equation}\label{g.2}
      \min_{y \in \lpi(\ltml)}(v_L(y))= v_L(\lpi(\varpi))=q^{\gamma}-\gamma v_L(\pi).
    \end{equation}
    In particular, the smallest disc containing $\lpi(\ltml)$ is $\{x \in L \mid v_L(x) \geq q^{\gamma}-\gamma v_L(\pi)\}$.
\end{theorem}
\begin{proof}
Let $\gamma \geq 1$ be the unique integer satisfying (\ref{g.1}). By Theorem \ref{basisthm2} and Corollary \ref{dontcareaboutunit}, it suffices to show that 
$$\min\lb v_L(\lpi(\varpi^j)) \mid 1\leq j \leq \frac{qv_L(\pi)}{q-1}, \ q\nmid j\rb=v_L(\lpi(\varpi)).$$
Since $v_L(\varpi)=1\leq \displaystyle{\frac{v_L(\pi)}{q-1}}$, we have $\varpi\notin\dpil$, so by (\ref{ellxeq})
$$\ell_{\varpi}=\llf \log_q\lb \frac{v_L(\pi)}{q-1}\rb\rrf+1=\gamma. $$
In particular, it follows by Theorem \ref{genlogvallem} that
\begin{equation}\label{g.3}
    v_L(\lpi(\varpi))=q^\gamma-\gamma v_L(\pi).
\end{equation}
To prove (\ref{g.2}) it suffices by Theorem \ref{genlogvallem} and (\ref{g.3}) that we show
\begin{equation}\label{claim1}
    q^{\ell_{\varpi^j}}j-\ell_{\varpi^j}v_L(\pi) \geq q^\gamma-\gamma v_L(\pi) \text{ for } 1 \leq j \leq \frac{qv_L(\pi)}{q-1}.
\end{equation}
\begin{itemize}
    \item[\underline{Case 1:}] $\displaystyle{\frac{v_L(\pi)}{q-1}<j \leq \frac{qv_L(\pi)}{q-1}}$. Note that in this case $\ell_{\varpi^j}=0$ and (\ref{claim1}) becomes
\begin{equation}\label{claim1.1}
    j\geq q^\gamma- \gamma v_L(\pi) \text{ for all }j>\frac{v_L(\pi)}{q-1}.
\end{equation}
To prove (\ref{claim1.1}), we'll show
\begin{equation}\label{claim1.2}
    \frac{v_L(\pi)}{q-1}\geq q^\gamma-\gamma v_L(\pi).
\end{equation}
We have $\displaystyle{\frac{v_L(\pi)}{q-1} \geq q^{\gamma-1}}$ by (\ref{g.1}), and
\begin{align}\label{a.1}
    \frac{v_L(\pi)}{q-1} \geq q^{\gamma-1} &\implies \frac{v_L(\pi)}{q-1} +\gamma v_L(\pi)\geq q^{\gamma-1}+\gamma q^{\gamma-1}(q-1).
\end{align}
Also
\begin{align}\label{a.2}
 q^{\gamma-1}+\gamma q^{\gamma-1}(q-1)\geq q^\gamma &\iff  q(\gamma-1)\geq \gamma-1
\end{align}
and the right side of (\ref{a.2}) is clear for $\gamma \geq 1$. By (\ref{a.1}) and (\ref{a.2}),
$$ \frac{v_L(\pi)}{q-1} +\gamma v_L(\pi)\geq q^{\gamma-1}+\gamma q^{\gamma-1}(q-1) \geq q^\gamma,$$
which proves (\ref{claim1.2}). Since we are considering $j>\displaystyle{\frac{v_L(\pi)}{q-1}}$, (\ref{claim1.2}) proves (\ref{claim1.1}).\footnote{Actually, it proves $j>q^\gamma-\gamma v_L(\pi)$ for all $j>\frac{v_L(\pi)}{q-1}$, but this is stronger than we need for our purposes.}
The upshot is: we have now proved (\ref{claim1}) in Case $1$, so
\begin{equation}\label{claim1.3}
    v_L(\lpi(\varpi^j)) \geq v_L(\lpi(\varpi)) \text{ for all } \frac{v_L(\pi)}{q-1}< j\leq \frac{qv_L(\pi)}{q-1}.
\end{equation}
\item[\underline{Case 2:}] $1\leq j\leq \displaystyle{\frac{v_L(\pi)}{q-1}}$. Since (\ref{claim1}) is an equality when $j=1$, we may assume $1<j\leq \displaystyle{\frac{v_L(\pi)}{q-1}}$. For such $j$,
$$\ell_{\varpi^j}:=\llf \log_q\lb  \frac{v_L(\pi)}{j(q-1)}\rb \rrf +1$$
by (\ref{ellxeq}). In particular, $\ell_{\varpi^j}$ is the unique integer in $\{1,\ldots,\gamma\}$ such that
\begin{equation}\label{claim1.4}
    \frac{v_L(\pi)}{q^{\ell_{\varpi^j}}(q-1)}<j\leq \frac{v_L(\pi)}{q^{\ell_{\varpi^j}-1}(q-1)}.
\end{equation}
If $\ell_{\varpi^j}=\gamma$ then $j \geq 1$ implies that
$$q^{\ell_{\varpi^j}}j-\ell_{\varpi^j}v_L(\pi) =q^\gamma j-\gamma v_L(\pi)\geq q^\gamma-\gamma v_L(\pi),$$
thus proving (\ref{claim1}) for $1<j\leq \displaystyle{\frac{v_L(\pi)}{q-1}}$ when $\ell_{\varpi^j}=\gamma$. 
For $j$ in this range with $1\leq \ell_{\varpi^j}\leq \gamma-1$, we have
\begin{align*}
  q^{\ell_{\varpi^j}}j-\ell_{\varpi^j} v_L(\pi) &>\frac{v_L(\pi)}{q-1}-\ell_{\varpi^j} v_L(\pi) \text{ by  (\ref{claim1.4})} \\
 &\geq q^{\gamma-1}-(\gamma-1)v_L(\pi) \text{ by (\ref{g.1}) and }\ell_{\varpi^j}\leq \gamma-1\\
 &\geq q^\gamma-\gamma v_L(\pi) 
\end{align*}
since 
$$q^{\gamma-1}-(\gamma-1)v_L(\pi) \geq q^\gamma-\gamma v_L(\pi) \iff \frac{v_L(\pi)}{q-1} \geq q^{\gamma-1}.$$

We have now shown that
$$q^{\ell_{\varpi^j}}j-\ell_{\varpi^j} v_L(\pi) \geq q^\gamma- \gamma v_L(\pi) \text{ for all } 1\leq j \leq \frac{v_L(\pi)}{q-1},$$
which is equivalent to
\begin{equation}\label{claim1.5}
    v_L(\lpi(\varpi^j)) \geq v_L(\lpi(\varpi)) \text{ for all } 1\leq j \leq \frac{v_L(\pi)}{q-1}.
\end{equation}
\end{itemize}
Equations (\ref{claim1.3}) and (\ref{claim1.5}) together prove (\ref{g.2}).
\end{proof}
We conclude this section with an explicit example using Theorem \ref{minvalthm}.
\begin{example}
    Let $K=\qp$, and $[p](X)=(1+X)^p-1$ so that $\displaystyle{\log_{[p]}(X)=\log(1+X)}$
is the usual $p$-adic logarithm. Set $L:=\qp(\varpi)$ where $\varpi^{p+1}=p$. We will give a $\zp$-basis of $\log(1+\mml)$ and a precise lower bound on the elements of $v_L(\log(1+\mml))$ in this case. Since all extensions of $\q_2$ contain $-1$, we are only interested in $p>2$ here.

We will treat the cases $p=3$ and $p\geq 5$ separately.
\begin{itemize}
    \item Let us first assume $p \geq 5$. Then $L|\qp$ is $p$-regular by Lemma 4.2 since $v_L(p)=p+1$ and $p\geq 5$ implies that 
\begin{equation}\label{ex.1}
    1<\displaystyle{\frac{v_L(p)}{p-1}}<2.
\end{equation}
In particular, the ratio is not an integer. By Theorem \ref{basisthm2}, a $\zp$-basis of $\log(1+\mml)$ is given by 
$$\widetilde{B}_L=\{\log(1+\varpi^j)\mid 1\leq j \leq p+2, \ j\neq p\}.$$
In this example we will give another (more explicit) basis of $\log(1+\mml)$. The inequalities in (\ref{ex.1}) imply $\dpil=\mml^2$ so $\log(1+\mml^2)=\mml^2$. Since $(1+\mml)/(1+\mml^2) \cong \fp$, 
\begin{equation}\label{ex.2}
    1+\mml=\bigcup_{a^p=a}\lb (1+a\varpi)+(1+\mml^2)\rb 
\end{equation}
where the union is taken over a set of Teichm\"uller representatives for the residue field $\fp$. Take logs on both sides of (\ref{ex.2}) to get
$$\log(1+\mml)=\bigcup_{a^p=a}\lb \log((1+a\varpi))+\mml^2\rb$$
where for all $a \neq 0$,

\begin{align}\label{ex.5}
    \log(1+a\varpi) = \sum_{n\geq 1}\frac{(-1)^{n-1}(a\varpi)^n}{n} \equiv a\varpi+\frac{(a\varpi)^p}{p} \bmod{\mml^2}
\end{align}
since $n-v_L(n) \geq 2$ for all positive integers $n \not\in\{1,p\}$. With $a^p=a$ and $\varpi^{p+1}=p$ in (\ref{ex.5}),
$$\log(1+a\varpi)\equiv a\lb\varpi+\displaystyle{\frac1\varpi}\rb\bmod{\mml^2},$$
therefore
\begin{align*}
    \log(1+\mml)&= \bigcup_{a^p=a} \lb a\lb\varpi+\displaystyle{\frac1\varpi}\rb+\mml^2\rb.
\end{align*}
Note that $\log(1+\mml)/\mml^2$ is an additive group of order $p$ that is generated by the image of $\varpi+\displaystyle{\frac1\varpi}$, so we could instead write
\begin{equation}\label{ex.7}
    \log(1+\mml)=\bigcup_{b=0}^{p-1} \lb b\lb\varpi+\displaystyle{\frac1\varpi}\rb+\mml^2\rb=\zp  \lb\varpi+\displaystyle{\frac1\varpi}\rb+\mml^2.
\end{equation}
Since $O_L=\zp[\varpi]$, $\mml^2=\varpi^2 O_L$ as an additive group has $\zp$-basis given by
$\varpi^2,\varpi^3,\ldots, \varpi^{p+2}$. In particular, (\ref{ex.7}) can be rewritten as
\begin{equation}\label{ex.6}
    \log(1+\mml)=\zp\lb\varpi+\displaystyle{\frac1\varpi}\rb+\sum_{i=2}^{p+2}\zp\varpi^i.
\end{equation}
The right side of (\ref{ex.6}) can be refined: $\varpi^{p+1}=p$, and $p\cdot\lb\varpi+\displaystyle{\frac1\varpi}\rb=p\varpi+\varpi^p$. So, 
$$\varpi^{p+2}=p\varpi=p\lb\varpi+\displaystyle{\frac1\varpi}\rb-\varpi^p \in\zp\lb\varpi+\displaystyle{\frac1\varpi}\rb+\zp \varpi^p.$$
In particular, a refinement of (\ref{ex.6}) is
\begin{equation}\label{label}
\log(1+\mml)=\zp\lb\varpi+\displaystyle{\frac1\varpi}\rb+\sum_{i=2}^{p+1}\zp\varpi^i.
\end{equation}
Since $\varpi+\displaystyle{\frac1\varpi}, \varpi^2, \ldots,\varpi^{p+1}$ is a collection of $p+1$ elements which span the additive group $\log(1+\mml)$ as a $\zp$-module, they must form a basis since $\log(1+\mml)$ is a free $\zp$-module of rank $p+1$ by Theorem \ref{basisthm2}. By (\ref{label}) it is clear that the minimal valuation of the elements in $\log(1+\mml)$ is $-1$. We can deduce this from Theorem \ref{minvalthm} by noting that (\ref{ex.1}) implies that the $\gamma$ in Theorem \ref{minvalthm} is
$$\gamma:=\displaystyle{\left\lfloor\log_p\lb\frac{v_L(p)}{p-1}\rb \right\rfloor+1}=1, $$
since $p>2$. In particular, Theorem \ref{minvalthm} tells us the minimal valuation in $v_L(\log(1+\mml))$ is 
$$ v_L(\log(1+\varpi))=p-(p+1)=-1.$$
\item The case $p=3$ is similar to the previous case of $p\geq 5$. The difference in this case is that the additive group $\log(1+\mml)/\dpil$ has order 9, not 3. Let $\varpi=\sqrt[4]{3}$ so $L=\q_3(\sqrt[4]{3})$. Clearly $L$ contains the quadratic subfield $\q_3(\sqrt{3})$. If $L$ were to contain a primitive cube root of unity then it would also contain the subfield $\q_3(\sqrt{-3})$ and thus also $\q_3(i)$. But $\q_3(i)|\q_3$ is unramified whilst $L|\q_3$ is totally ramified. Thus $L|\q_3$ is $3$-regular. When $p=3$ we have
\begin{equation}\label{ex.8}
    \frac{v_L(p)}{p-1}=2 \implies \dpil=\mml^3
\end{equation}
so $\log(1+\mml^3)=\mml^3$. Working backwards from $\log(1+\mml^3)=\mml^3$ we will compute an explicit $\z_3$-basis of $\log(1+\mml)$ as before. Since $(1+\mml^2)/(1+\mml^3) \cong (1+\mml)/(1+\mml^2)\cong \mathbb{F}_3$, 
\begin{align}
\label{ex.9}   
\log(1+\mml)&=\bigcup_{-1\leq a \leq1}\lb \log(1+a\varpi)+\log(1+\mml^2)\rb\\
\label{ex.10} 
\log(1+\mml^2)&=\bigcup_{-1\leq a \leq1}\lb \log(1+a\varpi^2)+\mml^3\rb
\end{align}
by the same reasoning as before, where
\begin{align}
  \label{ex.11}  \log(1\pm \varpi) &\equiv (\pm\varpi) -\frac{(\pm\varpi)^2}{2}+\frac{(\pm\varpi)^3}{3}-\frac{(\pm\varpi)^6}{6}+\frac{(\pm\varpi)^9}{9}\bmod{\mml^3}, \text{ and} \\
   \label{ex.12} \log(1\pm \varpi^2) &\equiv (\pm\varpi^2)+\frac{(\pm \varpi^2)^3}{3}\equiv \mp\varpi^2 \bmod{\mml^3}.
\end{align}
The congruence in (\ref{ex.11}) follows from 
$n-v_L(n) \geq 3$ for all $n \notin \{1,2,3,6,9\}$ and the first congruence in (\ref{ex.12}) follows from $2n-v_L(n)\geq 3$ for all $n \not\in \{1,3\}$. The second congruence in (\ref{ex.12}) follows from $\varpi^4=3$ and $2\varpi^2 \equiv -\varpi^2\bmod{\mml^3}$. We can also refine (\ref{ex.11}): we have
\begin{align*}
    \log(1+\varpi)&\equiv \frac{1}{\varpi}-\varpi-\varpi^2\bmod{\mml^3}, \\
    \log(1-\varpi)&\equiv -\frac{1}{\varpi}+\varpi-\varpi^2 \bmod{\mml^3}.
\end{align*}
From this, and equations (\ref{ex.9}) and (\ref{ex.10}), we deduce that the group $\log(1+\mml)/\dpil$ has nine distinct coset representatives given by
$$0,\pm \varpi^2,\pm\lb\frac{1}{\varpi}-\varpi-\varpi^2\rb,\pm \lb-\frac{1}{\varpi}+\varpi-\varpi^2\rb, \pm \lb\frac{1}{\varpi}-\varpi\rb.$$
Over $\z_3$, these elements span the free module
\begin{equation}\label{ex.13}
    \z_3\varpi^2+\z_3\lb\frac1\varpi-\varpi\rb.
\end{equation}
As an additive $\z_3$-module, $\mml^3$ has basis $\varpi^3, \varpi^4, \varpi^5,\varpi^6$. Combine that with (\ref{ex.13}) to get
\begin{equation}\label{ex.14}
    \log(1+\mml)= \z_3\varpi^2+\z_3\lb\frac1\varpi-\varpi\rb+\sum_{j=3}^6\z_3\varpi^j.
\end{equation}
By Theorem \ref{basisthm2} $\log(1+\mml)$ is a free $\z_3$-module of rank $4$, but the right side of (\ref{ex.14}) has six summands. Since $\varpi^4=3$, we get $\varpi^6=3\varpi^2 \in \z_3\varpi^2$ and 
$$3\lb\frac1\varpi-\varpi\rb=\varpi^3-\varpi^5 \implies \varpi^5 \in \z_3\lb\frac1\varpi-\varpi\rb+\z_3\varpi^3.$$
So (\ref{ex.14}) can be refined:
$$ \log(1+\mml)= \z_3\varpi^2+\z_3\lb\frac1\varpi-\varpi\rb+\z_3\varpi^3+\z_3\varpi^4.$$
Thus we have found a basis. From this description of the image of the logarithm it is clear that the minimal valuation of the elements in $\log(1+\mml)$ is $-1$. We can also deduce this from Theorem \ref{minvalthm} by noting that 
$$\gamma:=\llf\log_3\lb2\rb\rrf+1=1$$
and the minimal valuation in $v_L(\log(1+\mml))$ is $v_L(\log(1+\varpi))=3-4=-1$.
\end{itemize}
\end{example}
\section{Nontrivial torsion and an $O_K$-basis of $\lpi(\ltmn)$}\label{extension}
    Whilst it is not true that $L=\q_2(\sqrt[3]{2})$ is a $2$-regular extension of $\q_2$ (due to $-1 \in \q_2$), one can show that the minimal element of $v_L(\log(1+\mml))$ is $-2$ (this is worked out in \cite{ontheimage}). In particular, this coincides exactly with what we would get if we tried to naively apply Theorem \ref{minvalthm}; in this \linebreak case $\gamma=\llf \log_2(3)\rrf+1=2$ implies that $2^\gamma-\gamma v_L(2)=4-6=-2$. 
 In this section we explore how some of the main results from Sections \ref{pireg} and \ref{main} can be extended to finite extensions $L|K$ that are not $\pi$-regular i.e., $K_\pi \subseteq L$. 
\begin{theorem}\label{genspanthm}
Let $L|K$ be a finite extension with $K_\pi \subseteq L$, uniformiser $\varpi \in O_L$, and residue field degree $f$.
\begin{enumerate}
    \item[(i)] The set of $[L:K]+f$ elements
$$\mathcal{S}_{L}:= \left \{\zeta_i \varpi^j \mid 1\leq i \leq f, \ 1\leq j <\frac{qv_L(\pi)}{q-1}, \ q\nmid j\right\} \cup \left \{ \zeta_i \varpi^{\frac{qv_L(\pi)}{q-1}}\mid 1\leq i \leq f \right\}$$
spans  $\ltml$ as an $O_K$-module, where the $\zeta_i$'s in $O_L$ are any lift of a $k$-basis for $k_L$.
    \item[(ii)] The set $\mathcal{S}_{L}$ is a minimal set of generators if and only if $L|K$ is totally ramified.
\end{enumerate}
\end{theorem}
\begin{proof}
To prove (i), one can adapt the proof of Theorem \ref{basisthm1}. As pointed out in Remark \ref{rmk410}, when $K_\pi \subseteq L$ we need only add back the level $j=\displaystyle{\frac{qv_L(\pi)}{q-1}}$ to the set in Theorem \ref{basisthm1} to obtain a spanning set of $\ltml$. After adding this level back into our generating system, we are left with the spanning set $\mathcal{S}_{L}$, where 
$$|\mathcal{S}_{L}|=f\lb \left\lfloor \frac{qv_L(\pi)}{q-1}\right\rfloor -  \left\lfloor \frac{v_L(\pi)}{q-1}\right\rfloor \rb +f =fv_L(\pi)+f=[L:K]+f.$$
For the second part of the theorem, recall that for any finite extension $L|K$, we have an isomorphism of $O_K$-modules $\ltml \cong \mup(L) \times O_K^{[L:K]}$ by Theorem \ref{decomplemma}. Since $K_\pi \subseteq L$, $\mup(L)=\mupn$ for some $n\geq 1$. By general Lubin-Tate theory, we know that $\mupn \cong O_K/\pi^n$. Hence, we can write 
$$\ltml\cong O_K/\pi^n \times O_K^{[L:K]}.$$
It follows by the structure theorem for finitely generated modules over a PID that the minimal number of generators of $\ltml$ as an $O_K$-module is $1+[L:K]$. We therefore conclude that $\mathcal{S}_L$ is a minimal spanning set if and only if $f=1$ (equivalently, $L|K$ is totally ramified).
\end{proof}

\begin{cor}\label{corspan}
For a finite extension $L|K$ with $K_\pi \subseteq L$, uniformiser $\varpi \in O_L$, and residue field degree $f$, the set 
    $$\widetilde{\mathcal{S}}_L:=\left \{ \lpi(\zeta_i \varpi^j)\mid 1\leq i \leq f, \, 1\leq j < \frac{qv_L(\pi)}{q-1}, \ q\nmid j \right\}\cup \left\{ \lpi(\zeta_i\varpi^{\frac{qv_L(\pi)}{q-1}}) \mid 1\leq i \leq f\right\}$$
    spans the additive group $\lpi(\ltml)$ as an $O_K$-module, where the $\zeta_i$'s in $O_L$ are any lift of a $k$-basis for $k_L$.
\end{cor}
\begin{proof}
    The logarithm is surjective onto its image and will send any spanning set of $\ltml$ to a spanning set of $\lpi(\ltml)$.
\end{proof}
As a consequence of Corollary \ref{corspan}, 
\begin{equation}\label{cortocor}
    \min_{y \in \lpi(\ltml)}(v_L(y))=\min_{y \in \widetilde{S}_L}(v_L(y)).
\end{equation}
For $\pi$-regular extensions $L|K$, (\ref{logval}) in Theorem \ref{genlogvallem} gave a formula for the valuation of any element in $\lpi(\ltml)$, and we used this in Theorem \ref{minvalthm} to determine the minimal element in $v_L(\lpi(\ltml))$. For arbitrary finite extensions $L|K$, (\ref{logval}) may not be true: if $\lambda$ is a nontrivial zero of $\px$ and $L=K(\lambda)$, then $\lambda$ is prime in $O_L$ so $v_L(\lambda)=1$. Since $\lpi(\lambda)=0$ and $\lpi:\mathcal{F}(\dpil) \longrightarrow \dpil$ is injective, $\lambda \notin \dpil$. Blindly applying Theorem \ref{genlogvallem} gives
$$\ell_\lambda=\llf \log_q\lb \frac{q-1}{v_L(\lambda)(q-1)}\rb\rrf +1=\llf \log_q(1)\rrf+1=1,$$
so
$v_L(\lpi(\lambda))=q-v_L(\pi)=q-(q-1)=1$, contradicting  $\lpi(\lambda)=0 \implies v_L(\lpi(\lambda))=\infty$.

To give a correct generalisation of Theorem \ref{genlogvallem}, we need only replace equality in (\ref{logval}) with `$\geq$': 

\begin{theorem}\label{genvalthm2}
For a finite extension $L|K$ with $x \in \mml$, let $\ell_x$ be the smallest nonnegative integer satisfying $[\pi^{\ell_x}](x) \in \dpil$. Then
\begin{equation}\label{logval2}
    v_L(\lpi(x)) \geq q^{\ell_x} v_L(x)-{\ell_x} v_L(\pi).
\end{equation}
Moreover,
\begin{gather*}
    \ell_x= \begin{cases}
        0 \quad \quad \quad \quad \quad \quad \,\,\quad \quad \quad \,\,\,\, \text{ if } x \in \dpil, \\
        \llf\log_q\lb \frac{v_L(\pi)}{v_L(x)(q-1)}\rb \rrf +1\,\,\,\, \text{ otherwise}
    \end{cases}
\end{gather*}
where $\log_q$ denotes the base $q$ logarithm in $\real$. 
\end{theorem}
\begin{proof}
Since the proof of Theorem \ref{genvalthm2} is very similar to that of Theorem \ref{genlogvallem}, we omit the details and provide only a sketch of the main idea.
 In the proof of Theorem \ref{genlogvallem} we showed for $x \in \mml \sm \dpil$ and $\ell_x$ as in the statement of the theorem that
\begin{equation}\label{rmkpf.0}
    v_L(\lpi(x))=v_L([\pi^{\ell_x}](x))- \ell_xv_L(\pi)
\end{equation}
and
\begin{equation}\label{rmkpf.1}
    v_L(x) \leq \frac{v_L(\pi)}{q^{\ell_x-1}(q-1)}
\end{equation}
so $v_L([\pi^{\ell_x}](x))=q^{\ell_x}v_L(x)$ by Remark \ref{rmk1} since $L|K$ is $\pi$-regular, thus proving the result. Whilst (\ref{rmkpf.0}) and (\ref{rmkpf.1}) did not depend on the $\pi$-regularity of $L|K$, the application of Remark \ref{rmk1} did: for arbitrary finite extensions $L|K$, Remark \ref{rmk1} implies $v_L([\pi^{\ell_x}](x))=q^{\ell_x}v_L(x)$ only when the inequality in (\ref{rmkpf.1}) is strict. In the case that we have equality in (\ref{rmkpf.1}), Remark \ref{rmk1} can only imply 
\begin{equation}\label{rmkpf.2}
    v_L([\pi^{\ell_x}](x))\geq q^{\ell_x}v_L(x).
\end{equation}
In any case, (\ref{rmkpf.2}) and (\ref{rmkpf.0}) prove 
(\ref{logval2}) if $x \notin \dpil$. The remaining details of the proof follow exactly as in Theorem \ref{minvalthm}. 
\end{proof}
Due to the inequality in (\ref{logval2}) not necessarily being equality when $K_\pi \subseteq L$, we can not necessarily pin down the element of minimal valuation in $\lpi(\ltml)$ for general finite extensions $L|K$. It is also not clear a priori how to whittle down the spanning set $\widetilde{\mathcal{S}}_L$ (Corollary \ref{corspan}) to get a basis of $\lpi(\ltml)$. For what remains, we will specialise to the case that $L$ is a Lubin-Tate extension of $K$. 

Fix $n\geq 1$ and a generator $\lambda$ of the cyclic $O_K$-module $\mupn$. Set $L=\kpn=K(\lambda)$. The following is a corollary of Theorem \ref{genspanthm}.

\begin{cor}\label{mincor}
    Fix $n\geq 1$ and a generator $\lambda$ of the $O_K$-module $\mupn$. The set of $q^n-q^{n-1}+1$ elements 
$$\mathcal{S}_{n}:= \left \{ \lambda^j \mid \ 1\leq j \leq q^n-1, q\nmid j\right  \} \cup \{\lambda^{q^n}\}$$
is a minimal set of generators for the $O_K$-module $\ltmn$.
\end{cor}
\begin{theorem}\label{ltbasisthm}
     Fix $n\geq 1$ and a generator $\lambda$ of the $O_K$-module $\mupn$.  The set of $q^n-q^{n-1}$ elements 
    $$\mathcal{B}_n:=\{\lpi(\lambda^j) \mid 2\leq j \leq q^n-1, \ q\nmid j\} \cup \{\lpi(\lambda^{q^n})\}$$
    is a basis of the additive group $\lpi(\ltmn)$ as an $O_K$-module. 
\end{theorem}
\begin{proof}
    Let $\mathcal{S}_n$ denote the spanning set of $\ltmn$ in Corollary \ref{mincor}. Since the logarithm is surjective onto its image, the set $$\lpi(\mathcal{S}_n)=\{\lpi(\lambda^j) \mid 1\leq j \leq q^n-1\} \cup \{\lpi(\lambda^{q^n})\}$$ spans the additive group $\lpi(\ltmn)$ as an $O_K$-module. Since $\lambda \in \mupn$, we have $\lpi(\lambda) =0$, so $j=1$ can be omitted from the spanning set $\lpi(\mathcal{S}_n)$ and what remains, $\mathcal{B}_n$, will still generate $\lpi(\ltmn)$ as an $O_K$-module. Since $|\mathcal{B}_n|=q^n-q^{n-1}=[\kpn:K]$ and $\lpi(\ltmn)$ is a free $O_K$-module of rank $[\kpn:K]$ by Theorem \ref{freelem}, it follows that $\mathcal{B}_n$ is a basis of $\lpi(\ltmn)$ as an $O_K$-module.
\end{proof}
We end with a simple example in the multiplicative $p$-adic setting.
\begin{example}
Set $K=\qp$ and $[p](X)=(1+X)^p-1$ so $\log_{[p]}(X)=\log(1+X)$ is the usual $p$-adic logarithm. Let $L=\qp(\zeta_p)$ where $\zeta_p$ is a primitive $p$-th root of unity and pick $\lambda=\zeta_p-1.$ By Theorem \ref{ltbasisthm}, the set of $p-1$ elements
$$\mathcal{B}:=\{\log(1+(\zeta_p-1)^j) \mid 2\leq j \leq p\}$$
is a basis of the additive group $\log(1+\mm_{\qp(\zeta_p)})$ as a $\zp$-module. Since $\log(1+\mm_{\qp(\zeta_p)})=\mm_{\qp(\zeta_p)}^2$ \cite{ontheimage, cycloimage}, it follows that $\mathcal{B}$ is a basis of $\mm_{\qp(\zeta_p)}^2$ as a $\zp$-module, equipped with its usual module structure. A standard basis of $\mm_{\qp(\zeta_p)}^2$ as a $\zp$-module is
\begin{equation}
  \label{standard}  \{ (\zeta_p-1)^j \mid 2\leq j \leq p\}.
\end{equation}
We can recover (\ref{standard}) from the basis $\mathcal{B}$ by noting that
\begin{align*}
    \log(1+(\zeta_p-1)^j) &=\sum_{n\geq 1}\frac{(-1)^{n-1}(\zeta_p-1)^{jn}}{n} \equiv(\zeta_p-1)^j \bmod{\mml^{j+1}}
\end{align*}
since the only exception to $jn-v_L(n) \geq j+1$ for $n \geq 1$ when $j\geq 2$ and $p\geq 2$ is when $n=1$. 
\end{example}

\end{document}